\documentclass[12pt]{article}

\usepackage{amsthm}
\usepackage[english]{babel}
\usepackage[utf8]{inputenc}
\usepackage{amsmath}
\usepackage{amsthm}
\usepackage{graphicx}
\usepackage[colorinlistoftodos]{todonotes}
\usepackage{wasysym}
\usepackage[numbers]{natbib}
\usepackage{babel}
\usepackage{bm}
\usepackage{graphicx}
\graphicspath{ {images/} }

\title{On Edge Dimension of a Graph}

\author{Nina Zubrilina}

\date{\today}
\begin{document}
\maketitle

\theoremstyle{definition}
\newtheorem{theorem}{Theorem}[section]
\newtheorem{lemma}[theorem]{Lemma}
\newtheorem*{claim}{Claim}
\newtheorem*{question}{Question}
\newtheorem{defn}[theorem]{Definition}
\newtheorem{remark}[theorem]{Remark}
\newtheorem{corollary}[theorem]{Corollary}
\newcommand{\diam}{\mathrm{diam}}
\newcommand{\edim}{\mathrm{edim}}
\newcommand\abs[1]{\left|#1\right|}
\newcommand{\tab}{\hspace*{1cm}}

\newcommand{\thistheoremname}{}
\newtheorem{genericthm}[theorem]{\thistheoremname}
\newenvironment{namedprob}[1]
  {\renewcommand{\thistheoremname}{#1}%
   \begin{genericthm}}
  {\end{genericthm}}
\newtheorem*{genericthm*}{\thistheoremname}
\newenvironment{namedprob*}[1]
  {\renewcommand{\thistheoremname}{#1}%
   \begin{genericthm*}}
  {\end{genericthm*}}

\begin{abstract}
Given a connected graph $G(V, E)$, the edge dimension, denoted $\edim(G)$, is the least size of a set $S \subset V$ that distinguishes every pair of edges of $G$, in the sense that the edges have pairwise different tuples of distances to the vertices of $S$. The notation was introduced by Kelenc, Tratnik, and Yero, and in their paper they asked several questions about some properties of $\edim$. In this article we answer two of these questions: we classify the graphs on $n$ vertices for which $\edim(G) = n-1$  and show that $\frac{\edim(G)}{\dim(G)}$ isn't bounded from above (here $\dim(G)$ is the standard metric dimension of $G$). We also compute $\edim (G\Box P_m)$ and $\edim(G + K_1)$.
\end{abstract}

\section{Introduction} 
 Let $G(V, E) $ be a simple unconnected graph. We define the distance between an edge $e = xy$ and vertex $v$ as: $$d(e, v) = \min \{ d(x, v), d(y, v) \}.$$ A vertex $v$ $\textit{distinguishes}$ two edges $e_1$ and $e_2$ if $d(e_1, v) \neq d(e_2, v)$. A set $S \subseteq V$ is an $\textit{edge metric generator}$ of a graph $G(V, E)$ if for any two distinct edges $ e_1, e_2 \in E$ there is a vertex $ s \in S $ such that $s$ distinguishes $e_1$ and $e_2$. An edge generating set with the smallest number of elements is called an $\textit{edge basis}$ of $G$, and the number of elements in an edge basis is the $\emph{edge dimension}$ of $G$ (denoted $\edim(G)$). 

This concept was introduced by Kelenc, Tratnik and Yero in \cite{2016arXiv160200291K} in analogy with the classical metric dimension $\dim(G)$ defined as follows: a vertex $v \in V$ distinguishes $v_1, v_2 \in V$ if $d(v, v_1) \neq d(v, v_2)$. A set $S \subseteq V$ is a \emph{vertex generating set} of $G$ if for any distinct $ v_1, v_2 \in V$ there is a vertex $ s \in S $ such that $s$ distinguishes $v_1$ and $ v_2$. A vertex generating set with the smallest number of elements is a \emph{vertex basis} of $G$, and the number of elements in a vertex basis is its \emph{dimension} (denoted $\dim(G))$.

Metric dimension was introduced by Slater in 1975 in \cite{slater1975leaves}, in connection with the problem of uniquely recognizing
the location of an intruder in a network. The same concept was introduced independently by Harary and Melter in \cite{harary1976metric}. This graph invariant is helpful in areas such as robot navigation (\cite{khuller1996landmarks}), chemistry (\cite{chartrand2000resolvability}, \cite{chartrand2000resolvability1}, \cite{johnson1993structure}) and problems of image processing and pattern recognition involving hierarchical data structures (\cite{melter1984metric}). Metric generators in graphs are also connected to coin weighing and the Mastermind game as discussed in \cite{caceres2007metric}.

In \cite{2016arXiv160200291K}, Kelenc, Tratnic and Yero introduce $\edim$ and calculate it for various graphs, including paths, cycles, trees and grids. They give examples of graphs for which $\edim(G)< \dim(G)$ (wheel graphs), $\edim(G)= \dim(G)$ (trees) and $\edim(G) > \dim(G)$ ($C_{4r}\Box C_{4t}$ for integers $r, t$). They give examples of graphs with $\frac{\edim(G)}{\dim(G)} \approx 5/2$ and ask if the  $\frac{\edim(G)}{\dim(G)}$ ratio is bounded from above. They also ask for the classification of the graphs with $\edim(G) = \abs{V} - 1$. In this paper we answer both questions. We also calculate $\edim(G \Box P_m)$ and $\edim(G + K_1)$. 

We will use the following notation: \\
Consider some vertex $x$ of a graph. The $\emph{distance tuple}$ of $x$ on $S \subseteq V$, $S = \{ v_1,\ldots, v_k \}$ is the tuple $$d_S(x) = (d(x, v_1), d(x, v_2), \ldots, d(x, v_k)).$$ It is easy to see that $S$ is a vertex generator if and only if the distance tuples on $S$ are different for all vertices of $V(G)$.\\
We define the distance tuple identically if $x$ is an edge. Similarly, $S$ is an edge generator if and only if the distance tuples on $S$ are different for all edges of $E(G)$. \\
We use the notation $N(v)$ for vertices adjacent to $v$ (not including $v$). We use $V(G)$ and $ E(G)$ to denote the vertices and edges of a graph $G$. We say $\diam(G)= \max\{d(u, v) \vert u, v \in V(G) \}$ and denote the maximal degree of the vertices of $G$ with $\Delta(G)$. We use notation $G_1 + G_2$ for the sum of graphs $G_1$, $G_2$, which is constructed by connecting all the vertices of $G_1$ with all the vertices of $G_2$. We use $P_m$ to denote a path of length $m$. We use $G_1 \Box G_2$ to denote the Cartesian product of $G_1$ and $G_2$. All the graphs are simple, connected and undirected.

\section{Graphs for which $\bm{\edim = \abs{V}-1}$}
 For a graph $G(V, E)$ it is easy to see that if $\abs{V} = n$, then $\edim \leq n-1$ as any $n-1$ vertices form an edge generating set. We will now describe all the graphs for which $\edim = n-1$. 

\begin{defn}
We call the set $(N(v_1) \cup N(v_2))\setminus ((N(v_1) \cap N(v_2))$ the $\emph{non-mutual neighbors}$ of $v_1$ and $v_2$.
\end{defn}

\begin{theorem}\label{ncondition}
Let $G(V, E)$ be a graph with $\abs{V} = n$. Then $\edim(G) = n - 1 $ if and only if for any distinct $ v_1, v_2 \in V $ there exists $ u \in V $ such that $v_1u \in E, v_2u \in E$ and $u$ is adjacent to all non-mutual neighbors of $v_1, v_2$.
\end{theorem}
\begin{proof}
Suppose  $\edim(G) = n-1$. Then for any distinct $ v_1 ,  v_2 \in V$, the set $  V \setminus{ \{ v_1, v_2 \} }$ doesn't generate the edges of $G$. Fix some $v_1$ and $v_2$ and let $S = V \setminus \{v_1, v_2 \}$. If $S$ doesn't generate the edges of $G$, there must exist two edges that have the same distances to all elements of $S$. Call them $e_1, e_2$. 

\begin{namedprob*}{Claim 1}
 Let $e_1 \neq e_2$ and $d_S(e_1) = d_S(e_2)$. Then $e_1 = v_1u$ and $e_2 = v_2u$ for some $u \in V$. 
\end{namedprob*}

\begin{proof}[Proof of claim 1] Suppose there is a vertex $v \in S$ such that $v$ is on exactly one of the two edges $e_1$ and $ e_2$. Then $v$ distinguishes $e_1$ and $e_2$ since it has distance $0$ to one of them and distance at least $1$ to the other. Thus since we assumed $S$ doesn't distinguish $e_1, e_2$, there can't be such a vertex in $S$. This means all the non-mutual vertices of $e_1, e_2$ must not be in $S$ (so must be in $\{v_1, v_2\}$). This is only possible if $e_1 = v_1u$, $e_2 = v_2u$ for some $u \in V$. This proves the claim. 
\end{proof}
Notice this property restricts G to having $\diam(G) \leq 2$, since we just showed for any choice of distinct $ v_1, v_2 \in V$ there is a $ u \in V $ such that $ v_1u \in E$ and $v_2u \in E$. Thus, $v_1u$ and $v_2u$ have distances $1$ or $2$ to all vertices in $S \setminus{ \{u \}} $. 
\begin{namedprob*}{Claim 2}
Let $e_1 = v_1u, e_2 = v_2u$, and say $d_S(e_1) = d_S(e_2)$. Then $u$ is connected to all nun-mutual neighbors of $e_1, e_2$. 
\end{namedprob*}
\begin{proof}[Proof of claim 2]
Consider a vertex $w \in S\setminus{\{u\}}$. Suppose $w$ is a non-mutual neighbor of $v_1, v_2$, so $wv_1 \in E, wv_2 \not \in E$. Since $w \in S$, by assumption $d(e_2, w) = d(e_1, w)$. Thus since $d(v_1, w) = 1$ and $d(v_2, w) = 2$, we must have $d(u, w) = 1$ (so $uw \in E$). The same holds if we switch $v_1$ and $v_2$. Thus $u$ must be a neighbor of all non-mutual neighbors of $v_1$ and $v_2$. 
\end{proof}
This proves that the stated condition is necessary. It is also sufficient: 
\begin{namedprob*}{Claim 3}
Let  $e_1 = v_1u$, $e_2 = v_2u$ and say $u$ is connected to all non-mutual neighbors of $v_1$ and $v_2$. Then $e_1$ and $e_2$ are indistinguishable by all vertices of $S$.
\end{namedprob*}
\begin{proof}[Proof of claim 3]
 Consider $w \in S$. If $w=u$, $d(e_1, w) = 0 = d(e_2, w)$. Otherwise, $w$ has distance $1 $ or $2$ to $e_1$ and $ e_2$. Say $d(w, e_1) = 1$. There are two cases: 
\begin{enumerate}
\item $ d(w, u) = 1$. Then obviously $d(w, e_2) = 1$.
\item$d(w, v_1) = 1$, $d(w, u) \neq 1$. We know $u$ has to be adjacent to all non-mutual neighbors of $e_1$ and $e_2$. We also know $u$ is not adjacent to $w$. This means $w$ can't be a non-mutual neighbor, so since $w$ is adjacent to $v_1$, $w$ also has to be adjacent to $v_2$. Thus $d(v_2, w) = 1$ and hence $d(e_2, w) = 1$. 
\end{enumerate}
This means that if one of the edges has distance $1$ to $w$, then so does the other. Since we already know the distances from these edges to elements of $S \setminus{\{u\}}$ can only be $1$ or $2$, this proves that $e_1$ and $e_2$ are equidistant from all elements of $S$. 
\end{proof}
This proves the theorem. 
\end{proof}

\begin{corollary} 
Let $G$ be a graph on $n$ vertices. Suppose $\edim(G) = n-1$. Then $\diam(G) \leq 2$ and every edge is in a cycle of length $3$. \end{corollary}

\begin{proof}
Theorem \ref{ncondition} implies that for any $v_1 \neq v_2 $ there is a $ u \in V $ such that $v_1u \in E$ and $v_2u \in E$, so $\diam(G) \leq 2$. Moreover, for any $xy \in E$ there exists $u\in V$ such that $ xu \in E$ and $yu \in E$. This means $xy$ is in a cycle $xuy$ of length $3$. 
\end{proof}

\section{The $\bm{\edim(G)}$ to $\bm{\dim(G)}$ ratio}
 A natural question that arises in the study of the edge dimension is how it is related to the dimension of the same graph.
\begin{question} For what triples $(x, y, n)$ does there exist a graph $G$ with $\dim(G) = x, \edim(G) = y$ and $\abs{V} = n$?
\end{question}
Kelenc, Tratnik and Yero give examples of graphs for which $\dim(G)<\edim(G)$, $ \dim(G) = \edim(G)$, and $\dim(G)>\edim(G)$. Moreover, they show that there exist graphs realizing all triples $(x, y, n)$ such that $$1<x\leq y \leq 2x \leq n-2.$$
One of the questions they ask is whether $\frac{\edim(G)}{\dim(G)}$ is bounded from above. In this section we show it's not.
 
\begin{theorem}\label{rationotbounded}
$\frac{\edim(G)}{\dim(G)}$ is not bounded from above.
\end{theorem}
We prove this theorem by finding a graph $F_k$ with $\edim(F_k) = k + 2^k -2$, and $\dim(F_k) = k$. The graph $F_k$ is defined as follows: 
\begin{defn}
For a positive integer $k$, let $F_k$ be the graph on vertex set $A \cup B$, where $ B = \{b_1 \ldots b_k\}$ and $A = \{ a_S | S \subseteq B \}$. Let $b_i, b_j$ be adjacent for all $b_i, b_j \in B$ with $b_i \neq b_j$, and let  $a_S, a_T$ be adjacent for all $a_S, a_T \in A$ with $a_S \neq a_T$. For any $b_i \in B, a_S \in A$ let $b_i, a_S$ be adjacent if and only if $b_i \in S$. Notice $\abs{V(F_k)} = k + 2^k$.
\end{defn}
\begin{figure}[hp]
\centering

\includegraphics[scale = 0.15]{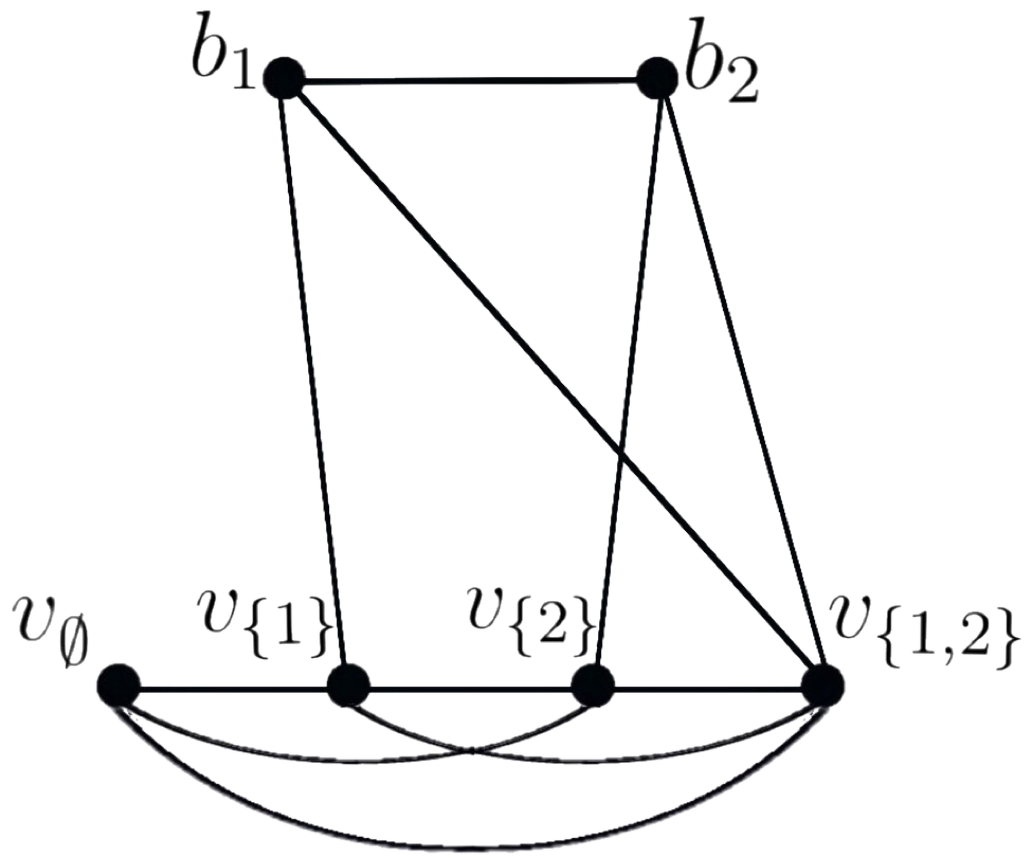}

\caption{ The graph $F_2$.}
\end{figure}

In order to determine some properties of $F_k$, we will use the following results.

\begin{lemma}[\cite{chartrand2000resolvability}]\label{diambound}
Let $G(V, E)$ be a graph with diameter $D$ and $\dim(G) = k$. Then $\abs{V} \leq k + D^k$.
\end{lemma}
We will prove this lemma to demonstrate the motivation for the construction of $F_k$.
\begin{proof}
Let $B = \{ b_1, \ldots, b_k \}$ be a vertex basis for $F_k$ and $V = \{ v_1, \ldots ,v_n \}$ be the vertices. Consider the distance tuples $d_B (v_i)$ for all $v_i \in V$. There are $k$ basis vertex tuples with exactly one $0$ in them (namely those of $b_1, \ldots, b_k$). All others tuples consist of $k$ numbers from $1$ to $D$. This shows there can be no more than $k + D^k$ different distance tuples. But the distance tuples have to be different for all vertices in order for $B$ to be a vertex basis. Thus, there can be no more than $k + D^k$ vertices.
\end{proof}

Below we prove an analogue of Lemma \ref{diambound} for edge dimension (we won't be using it for the proof of Theorem \ref{rationotbounded}).
\begin{theorem}
Let $G(V, E)$ be a simple connected graph with diameter $D$, $\abs{V} = n$, and $\edim(G) = k$. Then: $$\abs{E} \leq  \binom{k}{2} +  kD^{k-1} + D^k .$$
\end{theorem}
\begin{proof} 
Let $S$ be an edge basis. Consider the distance tuples on $S$ of the edges of $G$. There are at most $\binom{k}{2}$ distance tuples with two zeros (corresponding to the edges between pairs of vertices of $S$),  at most $kD^{k-1}$ tuples with one zero ($k$ ways to choose the position of the zero, $D^{k-1}$ options for the remaining places), and at most $D^k$ tuples with no zeros. Thus, since the tuples have to be different for all elements of $E$, we have $\abs{E} \leq  \binom{k}{2} +  kD^{k-1} + D^k$.
\end{proof}

\begin{lemma}[\cite{2016arXiv160200291K}]\label{largedegree}
Let $G(V, E)$ be a graph with $\abs{V} = n$ and $\Delta(G) = n-1$. Then:
$$\edim(G) = n-1 \text{ or } n-2.$$
\end{lemma}

\begin{lemma}[\cite{2016arXiv160200291K}]\label{twolargedegree}
Let $G(V, E)$ be a graph with $\abs{V} = n$ and  $\Delta(G) = n-1$. Suppose there are at least two vertices with degree $n-1$. Then: $$\edim(G) = n-1.$$
\end{lemma}
We will now use these lemmas to calculate $\dim(F_k)$ and $\edim(F_k)$. 
\begin{theorem}\label{propofgraph}
For any positive integer $k$, $$\dim(F_k)=k \text{ and }\edim(F_k) = k + 2^k - 2.$$ 
\end{theorem}

\begin{proof}
Since $a_B = a_{\{b_1, \ldots, b_k\}}$ is connected to all the other vertices of $F_k$,  $\diam(F_k) = 2$.  Since $\abs{V(F_k)} = k + 2^k$, Lemma \ref{diambound} guarantees $\dim(F_k) \geq k$. Moreover, $B$ is a vertex generating set since the distance tuples $d_B$ are different for all elements of $V(F_k)$ (this follows immediately from construction of $F_k$). Thus, $$\dim(F_k) = k.$$

Notice $a_B$ is connected to every vertex of $F_k$ by construction, so by Lemma \ref{largedegree} we know $\edim(F_k)$ is either $\abs{V(F_k)}- 1$ or $\abs{V(F_k)} -2$. Consider the vertices $a_{\emptyset}$ and $ a_B$. By construction of $F_k$ we know $a_{\emptyset}$ is not connected to any elements of $B$, and $a_B$ is connected to all of them. This means all elements of $B$ are non-mutual connections of $a_{\emptyset}$ and $a_B$. Also, notice that $a_B$ is the only vertex adjacent to all elements of $B$. This shows the condition of Theorem \ref{ncondition} doesn't hold for $F_k$, so 
$$\edim(F_k) = \abs{V(F_k)} - 2 = k + 2^k - 2.$$
\end{proof}
\begin{proof}[Proof of Theorem \ref{rationotbounded}]
By Theorem \ref{propofgraph}, $F_k$ is a counterexample to the boundedness of the $\edim(G)/\dim(G)$ ratio. 
\end{proof}
Another related question we could ask is the following:\\
Let $G(V, E)$ be a graph with $\abs{V} = n$ and $\edim(G) = n-1$. How large can $\dim(G)$ be? Consider the following example:
\begin{defn}
For a positive integer $k$, define $H_k$  = $F_k + K_1$.
\end{defn}
We will preserve the notation for the vertices of the subgraph $F_k$ of $H_k$ and call the $K_1$ vertex $t$.
\begin{theorem}
For any positive integer $k$, $$\dim(H_k) = k+1 \text{ and } \edim(H_k) = k + 2^k = n-1.$$
\end{theorem}
\begin{proof}
Due to Lemma \ref{diambound}, $\dim(H_k) \geq k+1$. We claim equality holds, and $B \cup \{ t\}$ is a vertex generating set. Indeed, consider any two vertices $x$ and $y$ in $V(H_k)$. If either of them is in $B \cup \{ t\}$, it distinguishes them. Otherwise, both $x$ and $y$ are in $A$. By construction of $F_k$, the vertices of $A$ have pairwise different distance tuples on $B$ consisting of $1's$ and $2's$. Notice that distance tuples of $A$ on $B$ are the same in $H_k$ as in $F_k$. Indeed, for $a \in A$ and $b \in B$, any path from $a$ to $b$ via $t$ will have length at least $2$, so can't be shorter than shorter than $d(a, b)$ in $F_k$. Hence, all pairs of vertices in $A$ are distinguished by $B$. This means $B \cup \{ t\}$ is a vertex generation set as claimed, so $\dim(H) = k +1$. 

Since $a_B$ and $t$ are connected to all the other vertices of $H_k$, by Lemma \ref{diambound}, $\edim(H_k) = \abs{V(H_k)} - 1 = k + 2^k$. 
\end{proof}

Recall that for a graph $G(V, E)$ with diameter $2$, Lemma \ref{diambound} implies that
 $$\dim(G) + 2^{\dim(G)} \geq \abs{V}.$$

In particular, in the case $\abs{V} = k + 2^k + 1$, this means that we can't make $\dim(G)$ smaller than $k + 1$. Since we showed in section $1$ that graphs $G(V, E)$ with edge dimension $\abs{V}-1$ have to have diameter $2$, this means we  cannot further decrease the dimension if we want the edge dimension to be maximal.

\section{$\bm{\edim}$ for $\bm{G + K_1}$ and $\bm{G \Box P_m}$}
In this section we characterize how the edge dimension changes upon taking a Cartesian product with a path, or upon adding a vertex a vertex adjacent to all the original vertices. 
\begin{theorem}
Let $G(V, E)$ be a graph with $\abs{V} = n$. Suppose for any vertex $ x \in V$ there is another vertex $ u \in V$ such that $V\setminus N(x)  \subseteq N(u)$. Then $\edim(G + K_1) = n$. Otherwise, $\edim(G + K_1) = n-1$.
\end{theorem}
\begin{proof}

Denote the $K_1$ graph vertex $t$. Since $t$ is connected to all the other vertices of $G + K_1$, by Lemma \ref{largedegree} $\edim(G + K_1)$ is either $n$ or $n-1$. We will use Theorem \ref{ncondition} to see when each case holds. Consider $x, y \in V$. Whatever their non-mutual connections are, $t$ is connected to all of them and to $x$ and $y$, so the hypothesis of Theorem \ref{ncondition} holds for this vertex pair. Now consider a pair $t$ and $ x \in V$. Their non-mutual neighbors are precisely $V\setminus  N(x) $. This means the condition stated in Theorem  \ref{ncondition} holds for $x, t$ if and only if there exists $u \in V$ such that $V\setminus N(x)  \subseteq N(u)$. Thus $\edim(G + K_1) = n$ if and only if this is true for any $x \in V$, which is what we were to prove. 
\end{proof}

\begin{theorem}
Let $G(V, E)$ be a graph and $P_m$ a path of length $m \geq 2$. Let $B_E \subseteq 2^V$ be the set of all the edge bases of $G$, let $B_V \subseteq 2^V$ be the set of all vertex bases of $G$. Let $k$ be the smallest possible cardinality of a union of an edge and a vertex basis, that is,  $$k = \min \big\{ \abs{S \cup T} \ \big\vert \  S\in B_V, T \in B_E \big \}.$$ Then: $$k \leq \edim(G\Box P_m) \leq k + 1.$$
\end{theorem}
\begin{proof}
Let $M = S \cup T$ with $S \in B_V, T \in B_E$ be a set for which the minimum cardinality is achieved, that is $\abs{M} = k$. \\
The graph $G\Box P_m$ can be constructed the following way:
First, take $m$ copies of $G$. Denote the $i^{\text{th}}$ copy $G(i)$. Denote the vertices of $G(i)$ with $v(i)$ for all $v \in V$. Then, connect $v(i)$ and $v(i+1)$ for all $v \in V$, $i \in \{ 1, \ldots, m-1 \}$. 
\begin{figure}[hp]
\centering

\includegraphics[scale = 0.15]{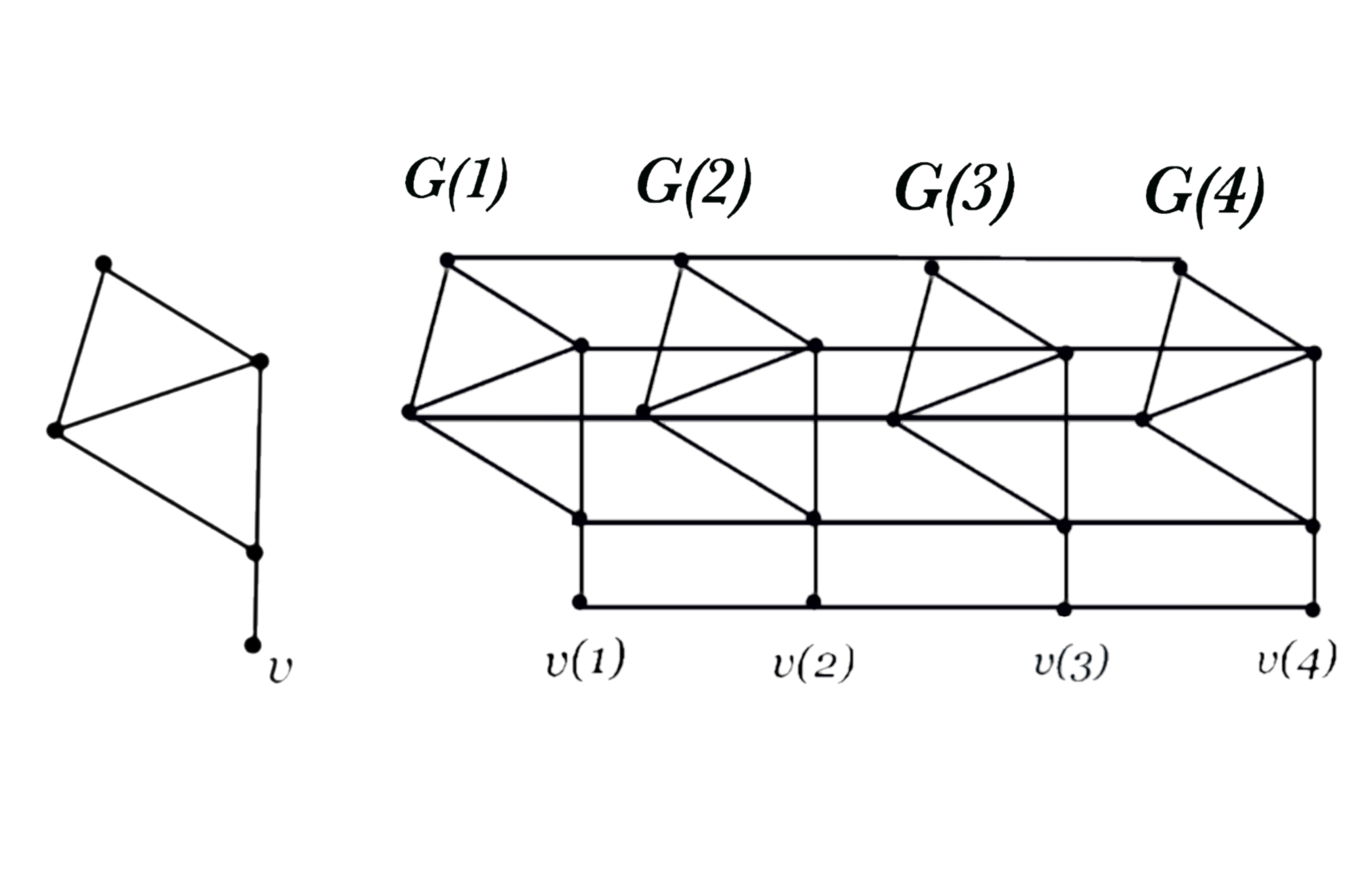}

\caption{ A graph $G$ and the described construction of the graph $G \Box P_4$.}
\end{figure}

\paragraph{Lower bound:}
Suppose $B$ is an edge basis of $G\Box P_m$. Let $B_1$ be the projection of $B$ on $G(1)$ (where we "project" $v(i)$ to $v(1)$). Consider $e \in E(G(1))$. Notice that $$d(e, v(i)) = d(e, v(1)) + i-1.$$ Thus, $e_1, e_2 \in E(G(1))$ are distinguished by $v(1)$ if and only if they are distinguished by $v(i)$. Thus, if $B$ is an edge generating set of $G\Box P_m$, then $B_1$ is an edge generating set of $G$. \\
Consider an edge $e = v(1)v(2)$. Notice that for $i \geq 2$ we have 
$$d(e, w(i)) = d(v(2), w(2)) + (i-2) = d(v(1), w(1)) + (i-2),$$ and for $i=1$, $$d(e, w(i)) = d(v(1), w(1)).$$ These differ by a constant only dependent on $i$. This means that if we consider two edges  $x = v(1)v(2)$ and  $y = u(1)u(2)$, then:
 $$w(i) \text{ distinguishes }x \text{ and } y \iff w(1) \text{ distinguishes }x \text{ and } y.$$ Moreover, $w(1)$ distinguishes $x$ and $ y$ if and only if $w(1)$ distinguishes $v(1)$ and $ u(1)$. Thus, $B_1$ is a vertex generating set of $G(1)$ as well.  
This shows that $B_1$ is both an edge generating set and a vertex generating set, so $\abs{B_1} \geq \abs{M}= k$. Also, clearly, $\abs{B_1} \leq \abs{B}$. This gives us the lower bound.

\paragraph{Upper bound:}
Let $M \subseteq V$ be a set defined in the statement of the theorem with $\abs{M}= k$, and let $t \in M$. Set $$B = \{v(1) \ \vert \ v \in M\} \cup t(m).$$ We will prove $B$ is an edge generating set of $G \Box P_m$. There are five cases of pairs of edges.
\begin{enumerate}
\item $e(i), f(i) \in E(G(i))$.\\
By definition of $M$, some $v \in M$ distinguishes $e(1), f(1)$. Since it's clear that $$d(v, z(i)) = d(v, z(1)) + i-1\text{ for any }z \in E,$$ $v$ also distinguishes $e(i)$ and $ f(i)$. 
\item $x(i)x(i+1)$ ; $y(i)y(i+1)$ for $x, y \in V$. \\
By definition of $M$, some $ v \in M$ distinguishes $x(1), y(1)$. Since 
$$d(v, z(i)) = d(v, z(1)) + i-1 \text{ for any }z \in V(G),$$ $v$ also distinguishes $x(i)$ and $ y(i)$. Thus $$d( v, x(i)x(i+1)) = d(v, x(i)) \neq d(v, y(i)) = d( v, y(i)y(i+1)).$$
\item$x(i)x(i+1)$,  $y(j)y(j+1)$ for $x, y \in V$ and $i \neq j$.\\
Notice that 
$$d(x(i)x(i+1), t(1)) = d(x(1), t(1)) + i-1$$ and 
$$d\big(x(i)x(i+1), t(m)\big) = d\big(x(m), t(m)\big) + m-i-1 =  d\big(x(1), t(1)\big) + m-i -1,$$ so
 $$d\big(x(i)x(i+1), t(1)\big) = d\big(x(i)x(i+1), t(m)\big) +m - 2i.$$ Thus, since we assumed $i \neq j$, we conclude that
if $t(1)$ doesn't distinguish $x(i)x(i+1), y(j)y(j+1)$, then $t(m)$ does.
\item $e(i), f(j)$ for $e, f \in E$, $i \neq j$. \\
Similarly to case $3$, we can see $$d\big(e(i), t(1)\big) = i-1 + d\big(e(1), t(1)\big),$$ and $$d\big(e(i), t(m)\big) = m-i + d\big(e(m), t(m)\big) = m-i + d\big(e(1), t(1)\big) = d\big(e(i)), t(1)\big) +m-2i+1.$$ Thus, if $t(1)$ does not distinguish $e(1)$ and $ f(j)$, then $t(m)$ does. 

\item$e(i)$, $y(j)y(j+1)$ for $e \in E$, $y \in V$. \\
 Suppose these two edges aren't distinguished by $t(1)$, so $$d\big(e(i), t(1)\big) = d\big(y(j)y(j+1), t(1)\big) = d.$$ As we have noted, then $$d(e(i), t(m)) = d +m - 2i +1 \text{ and } d(y(j)y(j+1), t(m)) = d +m - 2j.$$These can not be equal since they have different parity. 
\end{enumerate}
Since $\abs{B} = \abs{M} +1 = k +1$, this concludes the proof.
\end{proof}

\section{Conclusion and Open Problems}
 We have shown $\frac{\edim(G)}{\dim(G)}$ isn't bounded from above in section 3. More questions can be asked about the relationship between $\edim(G)$ and $\dim(G)$. For instance, 
\begin{itemize}
\item Are there graphs $G$ for which $\edim(G) \gg 2^{\dim(G)}$?
\item For what triples $x, y, n$ does there exist a graph $G$ with $\abs{V} = n$, $\dim(G) = x$ and $\edim(G) = y$? 
\end{itemize}
Another approach that could be taken to understand how $\dim(G)$ and $\edim(G)$ compare to each other is deriving some more properties of $\edim$ analogues to the known properties of $\dim$, as we did in the last sections 2 and 4. For example:
\begin{itemize}
\item For which graphs $G(V, E)$ is $\edim(G) = \abs{V}-2$?
\item For which graphs $G(V, E)$ is $\edim(G) = 2$?
\item For a graph $G$ and a positive integer $n$, bound $\edim(G\Box C_n)$ in terms of some function of $G$.
\item For graphs $G_1, G_2$, bound $ \edim(G_1 \Box G_2)$ in terms of some function of $G_1$ and $G_2$.
\end{itemize}

\section{Acknowledgments}
 The research was conducted during the Undergraduate Mathematics Research Program at University of Minnesota Duluth, and supported by grants NSF-1358659 and NSA H98230-16-1-0026. 
I would like to thank Joe Gallian for creating a marvelous working environment and for all his support and encouragements throughout the program. I would also like to thank Eric Riedl, Matthew Brennan and Levent Alpoge for very useful commentary and remarks.

\bibliography{Edgemetricdimension}
\bibliographystyle{plain}

\end{document}